\documentclass[a4paper,12pt]{amsart}

\usepackage{amsmath}
\usepackage{amssymb}
\usepackage{mathrsfs}
\usepackage{enumerate}
\usepackage{ifthen}
\usepackage{graphicx}
\usepackage[T1]{fontenc} %skandit
\usepackage{color}

\nonstopmode \numberwithin{equation}{section}
\newtheorem{thm}{Theorem}[section]
\newtheorem{cor}{Corollary}[section]
\newtheorem{lem}{Lemma}[section]

\theoremstyle{definition}

\newcounter{minutes}
\divide\time by 60
\newcounter{hours}
\multiply\time by 60
\addtocounter{minutes}{-\time}
\newcounter {own}
\def\theown {\thesection       .\arabic{own}}

{\qed\bigskip}

\newcounter{alphabet}

\begin{document}

\title{Schwarzian Norm Estimate for Functions in Robertson Class}

\author{Md Firoz Ali}
\address{Md Firoz Ali,
Department of Mathematics, National Institute of Technology Durgapur,
Durgapur- 713209, West Bengal, India.}
\email{ali.firoz89@gmail.com, firoz.ali@maths.nitdgp.ac.in}

\author{Sanjit Pal}
\address{Sanjit Pal,
Department of Mathematics, National Institute of Technology Durgapur,
Durgapur- 713209, West Bengal, India.}
\email{palsanjit6@gmail.com}


\def\thefootnote{}
\footnotetext{ {\tiny File:~\jobname.tex,
printed: \number\year-\number\month-\number\day,
          \thehours.\ifnum\theminutes<10{0}\fi\theminutes }
} \makeatletter\def\thefootnote{\@arabic\c@footnote}\makeatother

\begin{abstract}
Let $\mathcal{A}$ denote the class of analytic functions $f$ in the unit disk $\mathbb{D}=\{z\in\mathbb{C}:|z|<1\}$ normalized by $f(0)=0$, $f'(0)=1$. For $-\pi/2<\alpha<\pi/2$, let $\mathcal{S}_{\alpha}$ be the subclass of $\mathcal{A}$ consisting of functions $f$ that satisfy the relation  ${\rm Re\,} \{e^{i\alpha}(1+zf''(z)/f'(z))\}>0$ for $z\in\mathbb{D}$. In the present article, we determine the sharp estimate of the pre-Schwarzian and Schwarzian norms for functions in the class $\mathcal{S}_{\alpha}$.
\end{abstract}

\thanks{}

\maketitle
\pagestyle{myheadings}
\markboth{Md Firoz Ali and Sanjit Pal}{Schwarzian norm estimate of Robertson class}

\section{Introduction}
Let $\mathcal{H}$ denote the class of analytic functions in the unit disk $\mathbb{D}=\{z\in\mathbb{C}:|z|<1\}$, and  $\mathcal{LU}$ denote the subclass of $\mathcal{H}$ consisting of all locally univalent functions, namely, $\mathcal{LU}=\{f\in\mathcal{H}: f'(z)\ne 0, z\in\mathbb{D}\}$. The pre-Schwarzian and Schwarzian derivatives for a function $f\in\mathcal{LU}$ are defined by
$$P_f(z)=\frac{f''(z)}{f'(z)}\quad \text{and}\quad S_f(z)=\left(\frac{f''(z)}{f'(z)}\right)^{'}-\frac{1}{2}\left(\frac{f''(z)}{f'(z)}\right)^2,$$
respectively. Also, the pre-Schwarzian and Schwarzian norms (the hyperbolic sup-norms) of $f\in\mathcal{LU}$ are defined by
$$||P_f||=\sup\limits_{z\in\mathbb{D}}(1-|z|^2)|P_f
(z)| \quad \text{and}\quad ||S_f||=\sup\limits_{z\in\mathbb{D}}(1-|z|^2)^2|S_f(z)|,$$
respectively. These norms have significant meanings in the theory of Teichm\"{u}ller spaces (see \cite{Lehto-1987}). For a univalent function $f$ in $\mathcal{LU}$, it is well known that $||P_f||\leq 6$ and $||S_f||\leq 6$ (see \cite{Kraus-1932,Nehari-1949}) and these estimates are best possible. On the other hand, for a locally univalent function $f$ in $\mathcal{LU}$, it is also known that if $||P_f||\leq 1$ (see \cite{Becker-1972,Becker-Pommerenke-1984}) or $||S_f||\leq 2$ (see \cite{Nehari-1949}), then the function $f$ is univalent in $\mathbb{D}$. In 1976, Yamashita \cite{Yamashita-1976} proved that $||P_f ||$ is finite if and only if $f$ is uniformly locally univalent in $\mathbb{D}$. Furthermore, if $||P_f||<2$ then $f$ is bounded in $\mathbb{D}$ (see \cite{Kim-Sugawa-2002}). The pre-Schwarzian norm has also been studied by many researchers (see, for example, \cite{Ali-Sanjit-2022,Sugawa-1998,Yamashita-1999} and references therein).\\

Let us introduce one of the most important and useful tool known as differential subordination technique. In geometric function theory, many problem can be solved in a simple and sharp manner with the help of differential subordination. A function $f\in\mathcal{H}$ is said to be subordinate to another function $g\in\mathcal{H}$ if there exists an analytic function $\omega:\mathbb{D}\rightarrow\mathbb{D}$ with $w(0)=0$ such that $f(z)=g(\omega(z))$ and it is denoted by $f\prec g$. Moreover, when $g$ is univalent, then $f\prec g$ if and only if $f(0)=g(0)$ and $f(\mathbb{D})\subset g(\mathbb{D})$.\\

Let $\mathcal{A}$ denote the subclass of $\mathcal{H}$ consisting of functions $f$ with the normalized conditions $f(0)=f'(0)-1=0$. Thus, any function $f$ in $\mathcal{A}$ has the Taylor series expansion of the form
\begin{equation}\label{s-00001}
f(z)=z+\sum\limits_{n=2}^{\infty}a_nz^n.
\end{equation}
Let $\mathcal{S}$ be the set of all univalent functions $f\in\mathcal{A}$ in $\mathbb{D}$.
%A function $f\in\mathcal{A}$ is called starlike (respectively, convex) if the image $f(\mathbb{D})$ is a starlike domain with respect to the origin (respectively, convex). The classes of all starlike and convex functions in $\mathcal{S}$ are denoted by $\mathcal{S^*}$ and $\mathcal{C}$, respectively.
%It is well known that a function $f$ in $\mathcal{A}$ is starlike (respectively, convex) if and only if ${\rm Re\,} (zf'(z)/f(z)) > 0$ (respectively, ${\rm Re\,} (1 + zf''(z)/f'(z)) > 0$) for $z\in\mathbb{D}$.
A function $f\in\mathcal{A}$ is said to be starlike of order $\alpha$ ($0\le \alpha< 1$) if ${\rm Re\,} (zf'(z)/f(z)) > \alpha$ for $z\in\mathbb{D}$. A function $f\in\mathcal{A}$ is said to be convex of order $\alpha$ ($0\le \alpha< 1$) if  ${\rm Re\,} (1 + zf''(z)/f'(z))> \alpha$ for $z\in\mathbb{D}$. The classes of all starlike and convex functions of order $\alpha$ in $\mathcal{S}$ are denoted by $\mathcal{S}^*(\alpha)$ and $\mathcal{C}(\alpha)$, respectively. Clearly, a function $f$ in $\mathcal{A}$ belongs to $\mathcal{C}(\alpha)$ if and only if $zf'\in\mathcal{S}^*(\alpha)$. Moreover, $\mathcal{S}^*(0)=:\mathcal{S}^*$ and $\mathcal{C}(0)=:\mathcal{C}$ are the class of starlike and convex functions, respectively. For further information on these classes, we refer to \cite{Duren-1983,Goodman-book-1983}.\\

A domain $\Omega$ containing the origin is called $\alpha$-spirallike if for each point $w_0$ in $\Omega$ the arc of the $\alpha$-spiral from origin to the point $w_0$ entirely lies in $\Omega.$ A function $f\in\mathcal{A}$ is said to be an $\alpha$-spirallike if
$${\rm Re\,}\left(e^{i\alpha}\frac{zf'(z)}{f(z)}\right)>0\quad\text{for }z\in\mathbb{D},$$
where $|\alpha|<\pi/2$. In 1933, \v{S}pa\v{c}ek \cite{Spacek-1933} introduced and studied the class of $\alpha$-spirallike functions and this class is denoted by $\mathcal{SP}(\alpha)$. Later on, Robertson \cite{Robertson-1969} introduced a new class of functions, denoted by $\mathcal{S}_{\alpha}$, in connection with $\alpha$-spirallike functions. A function $f\in \mathcal{A}$ is in the class $\mathcal{S}_{\alpha}$ if and only if
$${\rm Re\,} \left\{e^{i\alpha}\left(1+\frac{zf''(z)}{f'(z)}\right)\right\}>0\quad\text{for }z\in\mathbb{D}.$$
We note that $f\in \mathcal{A}$ is in the class $\mathcal{S}_{\alpha}$ if and only if $zf'(z)\in \mathcal{SP}(\alpha)$.  In terms of subordination, for a function $f\in\mathcal{A}$, we have
\begin{equation}\label{s-00010}
f\in\mathcal{S}_{\alpha}\iff e^{i\alpha}\left(1+\frac{zf''(z)}{f'(z)}\right)\prec \frac{e^{i\alpha}+e^{-i\alpha}z}{1-z}.
\end{equation}
For $\alpha=0$, the class $\mathcal{S}_{\alpha}$ reduces to classical class of convex functions. For general values of $\alpha$, Robertson \cite{Robertson-1969} proved that functions in the class $\mathcal{S}_{\alpha}$ need not be univalent in $\mathbb{D}$. By using Nehari's test, Robertson \cite{Robertson-1969} also proved that functions in the class $\mathcal{S}_{\alpha}$ is univalent in $\mathbb{D}$ if $\alpha$ satisfies the inequality $0<\cos{\alpha}\le x_0\approx  0.2034\cdots$, where $x_0$ denotes the positive root of the equation
$$16x^3+16x^2+x-1=0.$$
Later on, Libera and Ziegler \cite{Libera-Ziegler-1972} improved the range of $\alpha$ to $0<\cos{\alpha}\le 0.2564\cdots$ in 1972. In 1975, Chichra \cite{Chichra-1975} improved this range to $0<\cos{\alpha}\le 0.2588\cdots$. Surprisingly, in the same year, Pfaltzgraff \cite{Pfaltzgraff-1975} proved that functions in $\mathcal{S}_{\alpha}$ are univalent if $0<\cos{\alpha}\le 1/2$. Till now, this is the best improvement result of $\alpha$ for which functions in $\mathcal{S}_{\alpha}$ are univalent in $\mathbb{D}$. In 1977, Singh and Chichra \cite{Singh-Chichra-1977} proved that if $f\in\mathcal{S}_{\alpha}$ with $f''(0)=0$, then the function $f$ is univalent in $\mathbb{D}$ for all values of $\alpha$, where $|\alpha|<\pi/2$.\\

The pre-Schwarzian as well as Schwarzian derivatives are popular tools for studying the geometric properties of analytic mappings. But because of the computational difficulty, the pre-Schwarzian norm has received more attention than the Schwarzian norm. Although the classical work on the Schwarzian derivative in relation to the geometric function theory was done in \cite{Kuhnau-1971, Nehari-1949}, for various subclasses of locally univalent functions, not much more research has been done on Schwarzian  derivative. In connection with Teichm\"{u}ller spaces, estimating the Schwarzian norm for typical subclasses of locally univalent functions is an interesting problem. For the class of convex functions $\mathcal{C}$, the Schwarzian norm satisfies $||S_f||\le 2$ and the estimate is sharp. This result was proved repeatedly by many researchers (see \cite{Lehto-1977, Nehari-1976, Robertson-1969}). In 1996, Suita \cite{Suita-1996} studied the class $\mathcal{C}(\alpha)$, $0\le \alpha< 1$ and proved that the Schwarzian norm satisfies the following sharp inequality
$$||S_f||\le
\begin{cases}
2 & \text{ if }~ 0\le \alpha\le 1/2,\\
8\alpha(1-\alpha) & \text{ if }~ 1/2< \alpha<1.
\end{cases}
$$

A function $f\in\mathcal{A}$ is said to be strongly starlike (respectively, strongly convex) of order $\alpha$, $0<\alpha\le 1$ if $|\arg\{zf'(z)/f(z)\}|<\pi\alpha/2$ (respectively, $|\arg\{1+zf''(z)/f'(z)\}|<\pi\alpha/2$) for $z\in\mathbb{D}$. The classes of all strongly starlike and strongly convex functions of order $\alpha$ are denoted by $\mathcal{S}^*_{\alpha}$ and $\mathcal{K}_{\alpha}$, respectively. Chiang \cite{Chiang-1991} studied the Schwarzian norm for the class $\mathcal{S}^*_{\alpha}$ and proved the sharp inequality $||S_f||\le 6\sin(\pi\alpha/2)$. Later on, Kanas and Sugawa \cite{Kanas-Sugawa-2011} studied the Schwarzian norm for the class $\mathcal{K}_{\alpha}$ and proved the sharp inequality $||S_f||\leq 2\alpha$.\\

A function $f\in \mathcal{A}$ is said to be uniformly convex function if every circular arc (positively oriented) of the form $\{z\in\mathbb{D}: |z -\xi|=r\}$, $\xi\in\mathbb{D}$, $0<r<|\xi|+1$ is mapped by $f$ univalently onto a convex arc. The class of all uniformly convex functions is denoted by $\mathcal{UCV}$. In particular, $\mathcal{UCV}\subset\mathcal{C}$.
It is well known that (see \cite{Goodman-1991,Ma-Minda-1992}) a function $f\in \mathcal{A}$ is uniformly convex if and only if
$${\rm Re\,}\left(1+\frac{zf''(z)}{f'(z)}\right)> \left|\frac{zf''(z)}{f'(z)}\right|~~\text{for}~ z\in\mathbb{D}.$$
Kanas and Sugawa \cite{Kanas-Sugawa-2011} proved that the Schwarzian norm satisfies $||S_f||\le 8/\pi^2$ for $f\in\mathcal{UCV}$ and the estimate is sharp. In 2012, Bhowmik and Wirths \cite{Bhowmik-Wirths-2012} studied the class of concave functions $\mathcal{C}o(\alpha)$ for $1\le \alpha\le 2$ and obtained the sharp estimate $||S_f||\le 2(\alpha^2-1)$ for $f\in\mathcal{C}o(\alpha)$.  Recently, the present authors \cite{Ali-Sanjit-2022a} considered the classes of functions $\mathcal{G}(\beta)$ with $\beta>0$ and $\mathcal{F}(\alpha)$ with $-1/2\le \alpha\le 0$, consisting of functions in $\mathcal{A}$ that satisfy the relation ${\rm Re\,}\left(1+zf''(z)/f'(z)\right)<1+\beta/2$ for $z\in\mathbb{D}$, and ${\rm Re\,}\left(1+zf''(z)/f'(z)\right)>\alpha$ for $z\in\mathbb{D}$, respectively, and obtained the sharp estimates $||S_f||\le 2\beta(\beta+2)$ for $f\in\mathcal{G}(\beta)$ and $||S_f||\le 2(1-\alpha)/(1+\alpha)$ for $f\in\mathcal{F}(\alpha)$. A function $f\in\mathcal{A}$ is said to be  Janowski convex function if $1+zf''(z)/f'(z)\prec(1+Az)/(1+Bz)$, where $-1\le B<A\le 1$. The class of all Janowski convex functions is denoted by $\mathcal{C}(A,B)$ (see \cite{Janowski-1973a,Janowski-1973b}). Also, the present authors \cite{Ali-Sanjit-2022b} studied the class $\mathcal{C}(A,B)$ and obtained the sharp estimate of the Schwarzian norm for this class.\\

In this article, we mainly focus on the class $\mathcal{S}_{\alpha}$ and our aim is to find the estimate of the modulus of the Schwarzian derivative for functions in the class $\mathcal{S}_{\alpha}$ ($|\alpha|<\pi/2$). This result will yield a sharp estimate of the Schwarzian norm for functions in this class. Also, we determine the sharp estimate of the pre-Schwarzian norm for functions in the class  $\mathcal{S}_{\alpha}$, where $|\alpha|<\pi/2$.

\section{Main Results}

Let $\mathcal{B}$ be the class of analytic functions $\omega:\mathbb{D}\rightarrow\mathbb{D}$ and $\mathcal{B}_0$ be the class of Schwarz functions $\omega\in\mathcal{B}$ with $\omega(0)=0$. The Schwarz's lemma states that if $\omega\in\mathcal{B}_0$, then $|\omega(z)|\le |z|$ and $|\omega'(0)|\le 1$. The equality occurs in any one of the inequalities if and only if $\omega(z)=e^{i\alpha}z$, $\alpha\in\mathbb{R}$. An extension of Schwarz lemma, known as Schwarz-Pick lemma, yields the estimate $|\omega'(z)|\le (1-|\omega(z)|^2)/(1-|z|^2)$, $z\in\mathbb{D}$ when $\omega\in\mathcal{B}$. In 1931, Dieudonn\'{e} \cite{Dieudonne-1931} first obtain the precise region of variability of $\omega'(z_0)$ for a fixed $z_0\in\mathbb{D}$ over the class $\mathcal{B}_0$.

\begin{lem}[Dieudonn\'{e}'s lemma]\cite{Dieudonne-1931, Duren-1983}
Let $\omega\in\mathcal{B}_0$ and $z_0\ne 0$ be a fixed point in $\mathbb{D}$. The region of variability of $\omega'(z_0)$ is given by
\begin{equation}\label{s-00005}
\left|\omega'(z_0)-\frac{\omega(z_0)}{z_0}\right|\le \frac{|z_0|^2-|\omega(z_0)|^2}{|z_0|(1-|z_0|^2)}.
\end{equation}
Moreover, the equality occurs in \eqref{s-00005} if and only if $\omega$ is a Blaschke product of degree $2$ fixing $0$.
\end{lem}

The Dieudonn\'{e}'s lemma is an extension of Schwarz's lemma as well as Schwarz-Pick lemma. Here, we remark that a Blaschke product of degree $n\in\mathbb{N}$ is of the form
$$B(z)=e^{i\theta}\prod_{j=1}^{n}\frac{z-z_j}{1-\bar{z_j}z},\quad z,z_j\in\mathbb{D}, ~\theta\in\mathbb{R}.$$
The Dieudonn\'{e}'s lemma will play an important role to prove our main result and we also use this lemma to construct the extremal functions by using Blaschke product.\\

\begin{thm}\label{Thm-s-0005}
For $-\pi/2<\alpha<\pi/2$, let $f\in\mathcal{S}_{\alpha}$. The Schwarzian derivative $S_f$ satisfy the following inequalities:
\begin{enumerate}
\item[(i)] If  $|\alpha|\le\pi/6$ and $z\in\mathbb{D}$, then
\begin{equation}\label{s-00015}
|S_f(z)|\le \dfrac{2\cos{\alpha}(1-(1-|z|^2)\sin{|\alpha|})}{(1-|z^2|)^2(1-\sin{|\alpha|})}.
\end{equation}
\item[(ii)] If $|\alpha|>\pi/6$ and $z\in\mathbb{D}$, then
\begin{equation}\label{s-00020}
|S_f(z)|\le
\begin{cases}
\dfrac{2\cos{\alpha}(1-(1-|z|^2)\sin{|\alpha|})}{(1-|z^2|)^2(1-\sin{|\alpha|})}& ~\text{for}~|z|<\delta,\\[3mm]
\dfrac{2\cos{\alpha}\sin{|\alpha|}}{(1-|z|)^2}&~ \text{for}~|z|\ge\delta,
\end{cases}
\end{equation}
where
\begin{equation}\label{s-00025}
\delta=\frac{1-\sin{|\alpha|}}{\sin{|\alpha|}}.
\end{equation}
\end{enumerate}
\end{thm}

\begin{proof}
For $-\pi/2<\alpha<\pi/2$, let $f\in\mathcal{S}_{\alpha}$. Then from \eqref{s-00010}, we have
$$e^{i\alpha}\left(1+\frac{zf''(z)}{f'(z)}\right)\prec \frac{e^{i\alpha}+e^{-i\alpha}z}{1-z},$$
and so there exists an analytic function $\omega:\mathbb{D}\rightarrow\mathbb{D}$ with $\omega(0)=0$ such that
$$e^{i\alpha}\left(1+\frac{zf''(z)}{f'(z)}\right)= \frac{e^{i\alpha}+e^{-i\alpha}\omega(z)}{1-\omega(z)}.$$
A simple computation gives
\begin{equation}\label{s-00030}
\frac{f''(z)}{f'(z)}=\frac{2e^{-i\alpha}\cos{\alpha}\omega(z)}{z(1-\omega(z))},
\end{equation}
and therefore,
\begin{align}\label{s-00035}
S_f(z)&=\left(\frac{f''(z)}{f'(z)}\right)^{'}-\frac{1}{2}\left(\frac{f''(z)}{f'(z)}\right)^2\\
&=\frac{2e^{-i\alpha}\cos{\alpha}}{z^2}\left(\frac{z\omega'(z)-\omega(z)+ie^{-i\alpha}\sin{\alpha}\omega^2(z)}{(1-\omega(z))^2}\right).\nonumber
\end{align}
Let us consider the transformation $\displaystyle e^{-i(\alpha+\pi/2)}\zeta(z)=\omega'(z)-\frac{\omega(z)}{z}$. By \eqref{s-00005}, the function $\zeta$ varies over the closed disk
$$|\zeta(z)|\le \frac{|z|^2-|\omega(z)|^2}{|z|(1-|z|^2)}\quad\text{for fixed }|z|<1.$$
Using the transformation $\zeta(z)$ in \eqref{s-00035}, we obtain
\begin{equation}\label{s-00038}
S_f(z)=\frac{2ie^{-2i\alpha}\cos{\alpha}}{z^2}\left(\frac{-z\zeta(z)+\sin{\alpha}\omega^2(z)}{(1-\omega(z))^2}\right),
\end{equation}
and hence,
\begin{align*}
|S_f(z)|&\le 2\cos{\alpha}\left(\frac{|\zeta(z)|}{|z||1-\omega(z)|^2}+\frac{|\omega(z)|^2\sin{|\alpha|}}{|z|^2|1-\omega(z)|^2}\right)\\
&\le 2\cos{\alpha}\left(\frac{|z|^2-|\omega(z)|^2}{|z|^2(1-|z|^2)|1-\omega(z)|^2}+\frac{|\omega(z)|^2\sin{|\alpha|}}{|z|^2|1-\omega(z)|^2}\right).
\end{align*}
For $0\le s:=|\omega(z)|\le |z|<1$, we obtain
\begin{align}\label{s-00040}
|S_f(z)|&\le 2\cos{\alpha}\left(\frac{|z|^2-s^2}{|z|^2(1-|z|^2)(1-s)^2}+\frac{s^2\sin{|\alpha|}}{|z|^2(1-s)^2}\right)\\
&=2\cos{\alpha}\left(\frac{|z|^2-s^2(1-(1-|z|^2)\sin{|\alpha|})}{|z|^2(1-|z|^2)(1-s)^2}\right)=2\cos{\alpha}~ g(s),\nonumber
\end{align}
where
\begin{equation}\label{s-00045}
g(s)=\frac{|z|^2-s^2(1-(1-|z|^2)\sin{|\alpha|})}{|z|^2(1-|z|^2)(1-s)^2}, \quad 0\le s\le |z|<1.
\end{equation}
Now, we wish to find the maximum value of $g(s)$ in $[0,|z|]$. To do this, we have to find the critical value of $g(s)$ in $(0,|z|)$. A simple calculation gives,
\begin{equation*}
g'(s)=\frac{2(|z|^2-s(1-(1-|z|^2)\sin{|\alpha|}))}{|z|^2(1-|z|^2)(1-s)^3},
\end{equation*}
and so, $g'(s)=0$ yields
$$s=\frac{|z|^2}{1-(1-|z|^2)\sin{|\alpha|}}=:s_0(|z|).$$
Now, we have to check for what values of $\alpha$ the point $s_0(|z|)$ lies in $(0,|z|)$. We note that the inequality
\begin{equation}\label{s-00050}
s_0(|z|)=\frac{|z|^2}{1-(1-|z|^2)\sin{|\alpha|}}<|z|,
\end{equation}
holds if and only if $h(|z|)>0$, where
\begin{equation}\label{s-00053}
h(t)=\sin{|\alpha|}t^2-t+1-\sin{|\alpha|}.
\end{equation}
The roots of $h(t)=0$ are given by $t_1=\delta=(1-\sin{|\alpha|})/\sin{|\alpha|}$ and $t_2=1$. The root $\delta$ lies in $(0,1)$ if and only if $|\alpha|>\pi/6$. We also note that $h(0)=1-\sin{|\alpha|}>0$.
%Moreover,
%$$h'(|z|)=2\sin{|\alpha|}|z|-1.$$
%If $\alpha=0$, then $h'(|z|)=-1<0$ and therefore, the function $h$ is strictly decreasing in $(0,1)$ and $h(1)=0$. Thus, $h$ has no zeros in $(0,1)$ and therefore, $h(|z|)>0$ in $(0,1)$. If $\alpha\ne 0$, then $h'(|z|)=0$ gives
%$$|z|=\frac{1}{2\sin{|\alpha|}},$$
%which lies in $(0,1)$ if and only if $|\alpha|>\pi/6$.
 Now, we complete the proof by considering two different cases.\\

\noindent\textbf{Case-I:} Let $|\alpha|\le \pi/6$. In this case, $h(|z|)>0$ in $(0,1)$ and so, $s_0(|z|)$ lies in $(0,|z|)$. Since, the numerator of $g'(s)$ is a linear function of $s$, and $g'(0)=2/(1-|z|)^2>0$, $g'(s_0(|z|))=0$, it follows that $g'(|z|)<0$. Therefore, the function $g$ is strictly increasing in $(0,s_0(|z|))$ and strictly decreasing in $(s_0(|z|),|z|)$ and therefore, $g$ attain its maximum at $s_0(|z|)$. Therefore, from \eqref{s-00040} and \eqref{s-00045}, we get
\begin{equation*}
|S_f(z)|\le \dfrac{2\cos{\alpha}(1-(1-|z|^2)\sin{|\alpha|})}{(1-|z^2|)^2(1-\sin{|\alpha|})}.
\end{equation*}

\noindent\textbf{Case-II:} Let $|\alpha|>\pi/6$. In this case, $h(|z|)>0$ for $|z|<\delta$ and $h(|z|)\le 0$ for $|z|\ge\delta$. Thus, for $|z|<\delta$, following the same argument as Case-I,
% Since $h'(0)<0$ and  $h'(1)>0$, and $h$ has a zero at $1$. This lead us to conclude that the  function $h$ has unique zero $\delta$ in $(0,1)$, which is given by \eqref{s-00025}.
the function $g$ is strictly increasing in $(0,s_0(|z|))$ and strictly decreasing in $(s_0(|z|),|z|)$ and therefore, $g$ attain its maximum at $s_0(|z|)$. Thus, the first inequality of \eqref{s-00020} follows from \eqref{s-00040} and \eqref{s-00045}. Since $h(|z|)\le 0$ for $|z|\ge \delta$, it follows that $\sin|\alpha| (1-|z|^2)\ge 1-|z|$. Thus, for $|z|\ge \delta$,  from \eqref{s-00045}, we obtain
 $$g(|z|)=\frac{\sin{|\alpha|}}{(1-|z|)^2} = \frac{\sin{|\alpha|}}{(1-|z|)^2}(1-|z|^2)g(0) \ge \frac{g(0)}{1-|z|}\ge g(0).$$
Thus, the second inequality of \eqref{s-00020} follows from \eqref{s-00040} and \eqref{s-00045}.\\

\end{proof}

Before we move to our next result, we discuss the sharpness of the estimate of $|S_f(z)|$ obtained in Theorem \ref{Thm-s-0005}.\\

Suppose that $\alpha$ and $z_0$ satisfy either of the following conditions:
\begin{enumerate}
\item [\textbf{(C1)}] \quad $|\alpha|\le \pi/6$ and $-1<z_0<1$,
\item [\textbf{(C2)}] \quad $|\alpha|> \pi/6$ and $-1<z_0<1$ with $|z_0|<\delta$, where $\delta$ is given by \eqref{s-00025}.
\end{enumerate}
Depending on $\alpha$, we choose a unimodular real number $p$ as follows:
\begin{align}\label{s-00055}
p=
\begin{cases}
-1 & \text{when}~ 0\le\alpha< \frac{\pi}{2},\\
1 & \text{when}~ -\frac{\pi}{2}<\alpha<0.
\end{cases}
\end{align}
For given $\alpha$ and $z_0$ satisfying \textbf{(C1)} or \textbf{(C2)}, we choose $p$ as above and consider the function $f_{z_0,p}$ defined by
\begin{equation}\label{s-00060}
e^{i\alpha}\left(1+\frac{zf_{z_0,p}''(z)}{f_{z_0,p}'(z)}\right)=\frac{e^{i\alpha}+e^{-i\alpha}\phi(z)}{1-\phi(z)},
\end{equation}
where
$$\phi(z)=\frac{pz(z-b)}{1-bz},$$
and $b$ is a solution of the equation
\begin{align}\label{s-00065}
&\frac{pz_0(z_0-b)}{1-bz_0}=\frac{z_0^2}{1+p(1-z_0^2)\sin{\alpha}}\nonumber\\
\text{i.e.,} \quad &b=\frac{z_0(1-p-\sin{\alpha}+z_0^2\sin{\alpha})}{-p+z_0^2-\sin{\alpha}+z_0^2\sin{\alpha}}.
\end{align}
We know that when $\alpha$ and $z_0$ satisfy \textbf{(C1)} or \textbf{(C2)}, the point $s_0(|z_0|)$ lies in $[0,|z|)$, where $s_0(|z_0|)$ is given by \eqref{s-00050} and therefore, $s_0(|z_0|)=|z_0|^2/(1+p(1-|z_0|^2)\sin{\alpha})$ lies in $[0,1)$. This lead us to conclude that $b$ lies in $(-1,1)$ and $\phi$ is a Blaschke product of degree $2$ fixing $0$. Therefore,  $f_{z_0,p}$ belongs to the class $\mathcal{S}_{\alpha}$. Thus, from \eqref{s-00038}, the Schwarzian derivative of $f_{z_0,p}$ is given by
\begin{align*}
S_{f_{z_0,p}}(z)&=\left(\frac{f''_{z_0,p}(z)}{f'_{z_0,p}(z)}\right)^{'}-\frac{1}{2}\left(\frac{f''_{z_0,p}(z)}{f'_{z_0,p}(z)}\right)^2\\
&=2ie^{-2i\alpha}\cos{\alpha}\left(\frac{-z\zeta(z)+\sin{\alpha}\phi^2(z)}{z^2(1-\phi(z))^2}\right),
\end{align*}
where $\zeta(z)$ is given by
$$\zeta(z)=e^{i(\alpha+\pi/2)}\left(\phi'(z)-\frac{\phi(z)}{z}\right).$$
Therefore, the Schwarzian derivative $S_{f_{z_0,p}}(z)$ at $z_0$ is given by
\begin{align*}
S_{f_{z_0,p}}(z_0)&=2ie^{-2i\alpha}\cos{\alpha}\left(\frac{-z_0\zeta(z_0)+\sin{\alpha}\phi^2(z_0)}{z_0^2(1-\phi(z_0))^2}\right)\\
&=2ie^{-2i\alpha}\cos{\alpha}\left(\frac{-p(1-b^2)+\sin{\alpha}(z_0-b)^2}{(1-bz_0-pz_0(z_0-b))^2}\right).
\end{align*}
Substituting the value of $b$ given in \eqref{s-00065} and calculating the Schwarzian derivative $f_{z_0,p}(z_0)$ by using $p^2=1$, we get
$$S_{f_{z_0,p}}(z_0)=2ie^{-2i\alpha}\cos{\alpha}\left(\frac{p+(2-z_0^2)\sin{\alpha}+p(1-z_0^2)\sin^2{\alpha}}{(1-z_0^2)^2(1+p\sin{\alpha})^2}\right).$$
Thus, for any $p$ given in \eqref{s-00055}, we have
\begin{equation}\label{s-00070}
|S_{f_{z_0,p}}(z_0)|=\dfrac{2\cos{\alpha}(1-(1-|z_0|^2)\sin{|\alpha|})}{(1-|z_0|^2)^2(1-\sin{|\alpha|})}.
\end{equation}
This shows that the inequality \eqref{s-00015} and the first inequality of \eqref{s-00020} are sharp for real $z$.\\

Now, we discuss the sharpness of the second inequality of \eqref{s-00020}. To do this, let us consider the function $f_0(z)$ given by
\begin{equation}\label{s-00075}
e^{i\alpha}\left(1+\frac{zf_0''(z)}{f_0'(z)}\right)=\frac{e^{i\alpha}+e^{-i\alpha}z}{1-z}.
\end{equation}
Using \eqref{s-00038} with $\omega(z)=z$, the Schwarzian derivative of $f_0(z)$ is given by
$$S_{f_0}(z)=\frac{2ie^{-2i\alpha}\sin{\alpha}\cos{\alpha}}{(1-z)^2}.$$
Therefore, for any $|\alpha|>\pi/6$ and $|z_0|\ge\delta$ with $0< z_0< 1$, we have
$$|S_{f_0}(z_0)|=\frac{2\cos{\alpha}\sin{|\alpha|}}{(1-|z_0|)^2},$$
which shows that the estimate of the second inequality of \eqref{s-00020} is sharp for real $z$.\\

The above discussion shows that the estimate of the Schwarzian derivative $|S_f(z)|$ obtained in Theorem \ref{Thm-s-0005} is sharp for certain real values of $z$, which will help us to prove the estimate of the Schwarzian norm $||S_f||$ is sharp for the functions in the class $\mathcal{S}_{\alpha}$.

\begin{thm}\label{Thm-s-00010}
For $-\pi/2<\alpha<\pi/2$, let $f\in\mathcal{S}_{\alpha}$. Then the Schwarzian norm of $f$ satisfy the following inequality
$$||S_f||\le
\begin{cases}
\dfrac{2\cos{\alpha}}{1-\sin{|\alpha|}}&~\text{if}~|\alpha|\le\dfrac{\pi}{6},\\[3mm]
8\cos{\alpha}\sin{|\alpha|} &~\text{if}~|\alpha|>\dfrac{\pi}{6}.
\end{cases}
$$
Moreover, the estimates are best possible.
\end{thm}

\begin{proof}
We prove the theorem by considering two different cases.\\

\noindent \textbf{Case-I: } Let $|\alpha|\le\pi/6$. Then from \eqref{s-00015}, the Schwarzian derivative of $f$ satisfies the following inequality
$$|S_f(z)|\le \dfrac{2\cos{\alpha}(1-(1-|z|^2)\sin{|\alpha|})}{(1-|z^2|)^2(1-\sin{|\alpha|})}.$$
Therefore,
\begin{align*}
||S_f||&=\sup\limits_{z\in\mathbb{D}}(1-|z|^2)^2|S_f(z)|\\
&\le 2\cos{\alpha}\sup\limits_{0\le |z|<1}\frac{1-(1-|z|^2)\sin{|\alpha|}}{1-\sin{|\alpha|}}\\
&=\frac{2\cos{\alpha}}{1-\sin{|\alpha|}}.
\end{align*}
To show that the estimate is best possible, we consider the function $f_{z_0,p}(z)\in\mathcal{S}_{\alpha}$, $-1<z_0<1$ defined by \eqref{s-00060}. From \eqref{s-00070}, we obtain
$$(1-|z_0|^2)^2|S_{f_{z_0,p}}(z_0)|=\frac{2\cos{\alpha}(1-(1-|z_0|^2)\sin{|\alpha|})}{1-\sin{|\alpha|}}\rightarrow \frac{2\cos{\alpha}}{1-\sin{|\alpha|}}\quad\text{as }z_0\rightarrow 1^-.$$
This shows that the estimate is best possible.\\

\noindent \textbf{Case-II: } Let $|\alpha|>\pi/6$. Then from \eqref{s-00020}, the Schwarzian derivative of $f$ satisfies the following inequality
\begin{equation*}
|S_f(z)|\le
\begin{cases}
\dfrac{2\cos{\alpha}(1-(1-|z|^2)\sin{|\alpha|})}{(1-|z^2|)^2(1-\sin{|\alpha|})}& ~\text{for}~|z|<\delta,\\[3mm]
\dfrac{2\cos{\alpha}\sin{|\alpha|}}{(1-|z|)^2}&~ \text{for}~|z|\ge\delta,
\end{cases}
\end{equation*}
where $\delta$ is given by \eqref{s-00025}. Therefore,
\begin{equation}\label{s-00080}
||S_f||=\sup\limits_{z\in\mathbb{D}}(1-|z|^2)^2|S_f(z)|=\max\{M_1,M_2\},
\end{equation}
where
\begin{align*}
M_1&=2\cos{\alpha}\sup\limits_{0\le |z|<\delta}\dfrac{1-(1-|z|^2)\sin{|\alpha|}}{1-\sin{|\alpha|}}\\
&=\frac{2\cos{\alpha}(1-(1-\delta^2)\sin{|\alpha|})}{1-\sin{|\alpha|}},
\end{align*}
and
\begin{align*}
M_2&=2\cos{\alpha}\sup\limits_{\delta\le|z|<1}(1+|z|)^2\sin{|\alpha|}\\&=8\cos{\alpha}\sin{|\alpha|}.
\end{align*}
Since $\delta$ is a root of the equation $h(|z|)=0$, and therefore, $1-(1-\delta^2)\sin{|\alpha|}=\delta$. A simple calculation gives
$$M_1=\frac{2\delta\cos{\alpha}}{1-\sin{|\alpha|}}=\frac{2\cos{\alpha}}{\sin{|\alpha|}}\le 4\cos{\alpha}<8\cos{\alpha}\sin{|\alpha|}=M_2.$$
Thus, from \eqref{s-00080}, we get the desired result.\\

Now, we show that the estimate is sharp. Let us consider the function $f_0\in\mathcal{S}_{\alpha}$ defined by \eqref{s-00075}. Then the Schwarzian derivative of $f_0$ is given by
$$S_{f_0}(z)=\frac{2\cos{\alpha}\sin{\alpha}}{(1-z)^2}.$$
Therefore,
$$||S_{f_0}||=\sup\limits_{z\in\mathbb{D}}(1-|z|^2)^2|S_{f_0}(z)|=2\cos{\alpha}\sin{|\alpha|}\sup\limits_{z\in\mathbb{D}}\frac{(1-|z|^2)^2}{|1-z|^2}.$$
On the positive real axis, we have
$$\sup\limits_{0\le t<1}\frac{(1-t^2)^2}{(1-t)^2}=4.$$
Thus,
$$||S_{f_0}||=8\cos{\alpha}\sin{|\alpha|}.$$
\end{proof}

For the particular value $\alpha=0$, one can obtain the Schwarzian norm for functions in the class $\mathcal{C}$, which was first obtained by Robertson \cite{Robertson-1969}.

\begin{cor}
If $f\in\mathcal{S}_0=:\mathcal{C}$, then the Schwarzian norm satisfies the sharp inequality $$||S_f||\le 2.$$
\end{cor}

In the next theorem, we estimate the pre-Schwarzian norm for functions in the class $\mathcal{S}_{\alpha}$.

\begin{thm}\label{Thm-s-00015}
For $-\pi/2<\alpha<\pi/2$, let $f\in\mathcal{S}_{\alpha}$. Then the pre-Schwarzian norm satisfy the following sharp inequality
$$||P_f||\le 4\cos{\alpha}.$$
\end{thm}
\begin{proof}
For $-\pi/2<\alpha<\pi/2$, let $f\in\mathcal{S}_{\alpha}$. Then from \eqref{s-00030}, the pre-Schwarzian derivative of $f$ is given by
$$P_f(z)=\frac{f''(z)}{f'(z)}=\frac{2e^{-i\alpha}\cos{\alpha}\omega(z)}{z(1-\omega(z))},$$
where $\omega\in\mathcal{B}_0$. Therefore,
\begin{align*}
||P_f||&=\sup\limits_{z\in\mathbb{D}}(1-|z|^2)|P_f(z)|\\&\le 2\cos{\alpha}\sup\limits_{0\le |z|<1}(1+|z|)\\
&=4\cos{\alpha}.
\end{align*}
It is easy to verify that the equality occur for the function $f_0\in \mathcal{S}_{\alpha}$ defined by \eqref{s-00075}.

\end{proof}

For the particular value $\alpha=0$, one can obtain the pre-Schwarzian norm for the functions in the class $\mathcal{C}$, which was first proved by Yamashita \cite{Yamashita-1999}.

\begin{cor}
If $f\in\mathcal{S}_0=:\mathcal{C}$, then the pre-Schwarzian norm satisfies the sharp inequality $$||P_f||\le 4.$$
\end{cor}

\vspace{4mm}
\textbf{Data availability:} Data sharing not applicable to this article as no data sets were generated or analyzed during the current study.\\

\textbf{Authors contributions:}  All authors contributed equally to the investigation of the problem and the order of the authors is given alphabetically according to the surname. All authors read and approved the final manuscript.\\

\textbf{Acknowledgement:} The second named author thanks the University Grants Commission for the financial support through UGC Fellowship (Grant No. MAY2018-429303).

\end{document}